\documentclass[12pt,reqno,oneside]{amsart} 
\usepackage[utf8]{inputenc} 

\usepackage{geometry} 
\allowdisplaybreaks
\usepackage{graphicx} 
\DeclareRobustCommand{\lbinom}{\genfrac{[}{]}{0pt}{}}

\usepackage{booktabs} 
\usepackage{array} 
\usepackage{paralist} 
\usepackage{verbatim} 
\usepackage{subfig} 
\usepackage{chngcntr}
\usepackage{amsmath}
\usepackage{amssymb}

\usepackage{fancyhdr} 
\def\um{{\underline m}}
\def\ul{{\underline l}}

\newcommand{\Lnum}[3]{{ #1  ;  #2 \brack  #3 }}

\renewcommand{\l}{\ensuremath{\lambda}}

\renewcommand{\phi}{\varphi}

\newcommand{\h}{\ensuremath{\dot{\mathfrak{h}}}}
\newcommand{\G}{\ensuremath{\mathfrak{g}}}

\newcommand{\MM}{\ensuremath{\mathsf{M}}}
\renewcommand{\H}{\ensuremath{\mathfrak{h}}}
\newcommand{\N}{\ensuremath{\mathfrak{N}}}

\newcommand{\VV}{\mathsf{V}}

\newcommand{\HH}{\mathsf{H}}
\newcommand{\WW}{\mathsf{W}}

\newcommand{\PP}{\mathsf{P}}
\newcommand{\NN}{\mathsf{N}}
\numberwithin{equation}{section}
\newtheorem{theorem}{Theorem}[section]
\newtheorem{proposition}[theorem]{Proposition}
\newtheorem{corollary}[theorem]{Corollary}
\newtheorem{lemma}[theorem]{Lemma}

\title{Quantum affine modules for non-twisted Affine Kac-Moody algebras}
\author{V. Futorny}
\author{J.T. Hartwig}
\author{E.A. Wilson}

\begin{document}

\begin{abstract}
We construct new irreducible weight modules over quantum affine algebras of type I with all weight spaces infinite-dimensional. These modules are obtained by parabolic induction from irreducible modules over the Heisenberg subalgebra.
\end{abstract}
\maketitle

\section*{Introduction}
One of the distinctive features of affine Kac-Moody algebras, in contrast to the finite dimensional semisimple Lie algebra case, is the existence of Borel subalgebras which are not Weyl group conjugate to the standard ones generated by the positive (resp. negative) simple generators (\cite{JK} and \cite{F1}). This leads to a very rich module theory, which includes weight modules all of whose weight spaces are infinite dimensional constructed by Benkart, Bekkert, Futorny and Kashuba, (\cite{BBFK}). In this paper we construct modules for untwisted affine quantum groups based on this construction which also have all their weight spaces infinite dimensional.

The main tool we will use in this construction is the Poincar\'{e} Birkhoff Witt type basis given in terms of the  Drinfeld generators of the quantum affine algebra (\cite{D}, see also \cite{B}). These generators contain a copy of the infinite rank quantized Heisenberg algebra, $G_q$, generated by the imaginary root vectors. We obtain an isomorphism between $G_q$ and an infinite rank Weyl algebra of which it is a trivial deformation. This allows us to apply the classification of Futorny, Grancharov, and Mazorchuk (\cite{FGM}) to obtain an ample supply of irreducible $G_q$ modules. We use parabolic induction on these modules to obtain modules of the full quantum affine algebra. If $\VV$ is an $\mathbb{A}$-admissible irreducible $G_q$-module, then we use the $\mathbb{A}$-form of the induced module to show that it is a deformation of an irreducible module for the affine Kac-Moody algebra (see \S 1.1 for the definition of $\mathbb{A}$). Moreover, as in Lusztig's result (\cite{L}) we see that the weight multiplicities are the same as in the $q=1$ case. In the final section, we apply the result of the previous section to the case of $\varphi$-imaginary Verma modules for $\varphi:\mathbb{N}\to \{\pm\}$ an arbitrary function. By \cite[Proposition 3.4]{BBFK} and its corollary we obtain conditions for all weight spaces of these modules to be infinite dimensional and irreducible. 

\section{Preliminaries}
Throughout this paper, $\mathbb{N}$ denotes the set of non-negative integers. Let $I=\{1,\dots, n\}$ and let $\mathfrak{g}=X_n^{(1)}$ be an untwisted affine type Kac-Moody algebra corresponding to a generalized affine Cartan matrix of type 1: $ A=(a_{ij})_{0\leq i,j\leq n}$.   The loop algebra construction (see \cite{K}) gives:
\begin{equation}
\mathfrak{g}=\left (\mathring{\mathfrak{g}}\otimes \mathbb{C}[t,t^{-1}]\right )\oplus \mathbb{C}c\oplus \mathbb{C}d
\end{equation}
where $\mathring{\mathfrak{g}}$ is a simple finite-dimensional Lie algebra over $\mathbb{C}$ of type $X_n$, $c$ is the central generator and $d$ is the degree derivation. Then $\H = \h \oplus \mathbb Cc
\oplus \mathbb Cd$ is a Cartan subalgebra of $\mathfrak g$, where $\h$ is a Cartan subalgebra of $\mathring{\mathfrak{g}}$.  Let $\mathcal D=(d_0,\dots,d_n)$ be a diagonal matrix with relatively prime integer entries such that the matrix $\mathcal D A$ is symmetric.
 Let $\mathring{\Delta}$ be the root system of $\mathring{\mathfrak{g}}$, and $\mathring{\Delta}^{\pm}$ be the sets of positive and negative roots with respect to some choice of simple roots $\alpha_1, \ldots, \alpha_n$. Then, the sets of positive (resp. negative) roots of $\mathfrak{g}$ are:
\begin{equation*}
\Delta^{\pm}=\{\alpha\pm k\delta\mid \alpha\in\mathring{\Delta}\cup \{0\}, k\in \mathbb{Z}_{>0}\}\cup \mathring{\Delta}^{\pm}.
\end{equation*}

For each $\alpha\in \Delta$ set $\G_\alpha = \{ x \in \G \mid  [h, x] = \alpha(h) x \,
{\text{ for all }} \, h \in \H\}$. Denote by $\mathring{Q}$ the free abelian group generated by $\alpha_1, \ldots, \alpha_n$, a root lattice of $\mathring{\mathfrak{g}}$. Let $Q:=  \mathring{Q}\oplus {\mathbb Z}\delta$ be the root lattice of $\mathfrak g$ and $P = \{ \lambda \in {\mathfrak h}^*\ |\ \lambda(h_i) \in {\mathbb Z}, i \in I, \lambda(d) \in {\mathbb Z}\}$ be the weight lattice, where $h_1, \ldots, h_n$ is a basis of $ \dot{\mathfrak h}$.

Unlike in the finite-dimensional case, $\Delta$ does not consist of a single Weyl group orbit. Instead, $\Delta=\Delta^{re}\cup \Delta^{im}$ where $\Delta^{re}$ is the set of all roots that are conjugate to a simple root $\alpha_i$ and $\Delta^{im}=\Delta\backslash \Delta^{re}$, called the sets of real and imaginary roots respectively. As seen in \cite{K}, for an untwisted affine type Kac-Moody algebra we have $\Delta^{re}=\{\alpha+k\delta\mid \alpha\in \mathring{\Delta}, k\in \mathbb{Z}\}$ and $\Delta^{im}=\{k\delta\mid k\in\mathbb{Z}\backslash\{0\}\}$ where $\delta$ is the indivisible imaginary root.

Consider the partition $\Delta = S \cup -S$ of the root system of $\mathfrak g$
where $S=\{\alpha+k\delta \mid \alpha\in \mathring{\Delta}^+, k\in \mathbb{Z}\} \cup (\Delta^{im})^+$. This is a non-standard partition of the root system
$\Delta$ in the sense that $S$ is not conjugated  to the sets
$\Delta^{\pm}$  by the Weyl group. The classification of closed partitions of the root system for affine Kac-Moody algebras was obtained by Jakobsen and Kac \cite{JK}, and independently by Futorny \cite{F1,F2}.

The {\it quantum affine algebra} $U_q(\mathfrak g)$ is the $\mathbb C(q^{1/2})$-algebra with 1 generated by
$$ 
E_i, \enspace F_i, \enspace K_\alpha,\enspace \gamma^{\pm 1/2}, \enspace D^{\pm 1} \quad i\in I \cup \{0\},\enspace \alpha\in Q,
$$
and defining relations:
\begin{align*}& DD^{-1}=D^{-1}D=K_iK_i^{-1}=K_i^{-1}K_i=\gamma^{1/2}\gamma^{-1/2}=1, \\
&[\gamma^{\pm 1/2},U_q(\mathfrak g)]=[D,K_i^{\pm1}]=[K_i,K_j]=0, \\
&(\gamma^{\pm 1/2})^2=K_\delta^{\pm 1},\\
& E_iF_j-F_jE_i = \delta_{ij}\frac{K_i-K_i^{-1}}{q_i-q_i^{-1}}, \\
& K_\alpha E_iK_\alpha^{-1}=q^{(\alpha|\alpha_i)}E_i, \ \ K_\alpha F_i K_\alpha^{-1} =q^{-(\alpha|\alpha_i)}F_i, \\
& DE_iD^{-1}=q^{\delta_{i,0}} E_i,\quad DF_iD^{-1}=q^{-\delta_{i,0}} F_i, \\
& \sum_{s=0}^{1-a_{ij}}(-1)^sE_i^{(1-a_{ij}-s)} E_jE_i^{(s)}=0= \sum_{s=0}^{1-a_{ij}}(-1)^sF_i^{(1-a_{ij}-s)} F_jF_i^{(s)}, \quad i\neq j.
\end{align*}
where
$$
q_i:=q^{d_i},\quad 
[n]_i= \frac{q^n_i-q^{-n}_i}{q_i-q^{-1}_i},\quad [n]_i!:=\prod_{k=1}^n[k]_i
$$
and 
$K_i=K_{\alpha_i}$, $E_i^{(s)}=E_i/([s]_i!)$ and $F_i^{(s)}=F_i/([s]_i!)$. We also define $${r \brack s}_i:=\frac{[r]_i!}{[s]_i![r-s+1]_i!}.$$

Let $U_q^+=U_q^+(\mathfrak g)$ (resp. $U_q^-=U_q^-(\mathfrak g)$) be the subalgebra of $U_q(\mathfrak g)$ generated by $E_i$ (resp. $F_i$), $i \in I$, and let $U_q^0=U_q^0(\mathfrak g)$ denote the subalgebra generated by
$K_i^{\pm 1}$ ($i \in I$), $\gamma^{\pm 1/2}$, and $D^{\pm 1}$.

We will also use Drinfeld realization of $U_q(\mathfrak g)$ \cite{D}. It can be generated over $\mathbb C (q^{1/2})$ by
$$  
x_{ir}^{\pm},\enspace h_{is}, \enspace K_i^{\pm 1}, \enspace \gamma^{\pm 1/2},D^{\pm 1} \enspace i\in I,r,s\in\mathbb Z, s\neq 0,
$$
subject to the following relations:
\begin{align}
 DD^{-1}&=D^{-1}D=K_iK_i^{-1}=K_i^{-1}K_i=\gamma^{1/2}\gamma^{-1/2}=1,\label{drinfeldfirst} \\
[\gamma^{\pm 1/2},U_q(\mathfrak g)]&=[D,K_i^{\pm 1}]=[K_i,K_j]=[K_i,h_{jk}]=0, \\
Dh_{ir}D^{-1}&=q^rh_{ir},\quad Dx_{ir}^{\pm}D^{-1}=q^rx_{ir}^{\pm},\\
K_ix_{jr}^{\pm}K_i^{-1} &= q_i^{\pm  (\alpha_i|\alpha_j)}x_{jr}^{\pm},   \\  
[h_{ik},h_{jl}]&=\delta_{k,-l} \frac{1}{k}[ka_{ij}]_i\frac{\gamma^k-\gamma^{-k}}{q_j-q_j^{-1}}\label{hs} \\
[h_{ik},x^{\pm}_{jl}]&= \pm \frac{1}{k}[ka_{ij}]_i\gamma^{\mp |k|/2}x^{\pm}_{j,k+l}, \label{axcommutator}  \\
    x^{\pm}_{i,k+1}x^{\pm}_{jl} &- q^{\pm (\alpha_i|\alpha_j)}
x^{\pm}_{jl}x^{\pm}_{i,k+1}\label{Serre}   \\
&= q^{\pm (\alpha_i|\alpha_j)}x^{\pm}_{ik}x^{\pm}_{j,l+1}
    - x^{\pm}_{j,l+1}x^{\pm}_{ik},\notag \\
[x^+_{ik},x^-_{jl}]&=\delta_{ij}
    \frac{1}{q_i-q^{-1}_i}\left( \gamma^{\frac{k-l}{2}}\psi_{i,k+l} -
    \gamma^{\frac{l-k}{2}}\phi_{i,k+l}\right), \label{xcommutator}   \\
\text{where  }
\sum_{k=0}^{\infty}\psi_{ik}z^{k} &= K_i \exp\left(
(q_i-q^{-1}_i)\sum_{l>0}  h_{il}z^{l}\right), \text{ and }\notag\\
\sum_{k=0}^{\infty}
\phi_{i,-k}z^{-k}&= K^{-1}_i \exp\left( - (q_i-q^{-1}_i)\sum_{l>0} 
h_{i,-l}z^{-l}\right).\label{phidef}\\
\text{For }i\neq j,\enspace N:=1-a_{ij} \notag  \\
\text{Sym}_{k_1,k_2,\dots,k_N}&\sum_{r=0}^{N}(-1)^r \genfrac{[}{]}{0pt}{}{N}{r}_i x_{ik_1}^\pm \cdots x_{ik_r}^\pm x_{jl}^\pm x_{ik_{r+1}}^\pm \cdots x_{ik_s}^\pm=0
\end{align}
 (see \cite{B} and \cite{CFM}). Here $\text{Sym}$ denotes symmetrization with respect to  $k_1,\dots, k_n$.  

Beck \cite{B}  has defined an ordering of $\Delta$ and a related PBW-type basis for quantum affine algebras. Let $\rho=(1/2)\sum_{\beta\in \mathring{\Delta}^+}\beta$ and fix a reduced expression $t_{\rho}=r_{i_1}r_{i_2}\cdots r_{i_d}$ for the Weyl group element $t_\rho$ corresponding to translation by $\rho$. For $\mathring{\mathfrak{g}}$ recall the Lusztig automorphisms (\cite{L}) $T_i,0\leq i \leq n$ of $U_q(\mathfrak{g})$. To each root $\alpha\in \Delta^{re}$ one assigns a corresponding root vector $E_\alpha$. Let $(i_k)_{k\in \mathbb{Z}}$ be the sequence of integers such that $i_k= i_{k\pmod{d}}$ and define the following sequence of real roots:
\begin{equation*}
\beta_k=\begin{cases}
r_{i_1}r_{i_2}\cdots r_{i_{k-1}}(\alpha_k), \text{ if }k\geq 0,\\
r_{i_0}r_{i_{-1}}\cdots r_{i_{k+1}}(\alpha_k), \text{ otherwise.}
\end{cases}
\end{equation*}
Then $(\Delta^{re})^+=(\beta_k)_{k\in \mathbb{Z}}$ and we put: 
\begin{equation*}
E_{\beta_k}=
\begin{cases}
T_{i_1}T_{i_2}\cdots T_{i_{k-1}}(F_{i_k})\text{ if }k\geq 0,\\
T_{i_0}^{-1}T_{i_{-1}}^{-1}\cdots T_{i_{k+1}}^{-1}(E_{i_k}), \text{ otherwise}.
\end{cases}
\end{equation*}
The positive imaginary root vectors are defined by the following functional equation:
\begin{equation*}
1+(q_i-q_i^{-1})\sum_{k\geq 0}K_i^{-1}[E_i,E_{-\alpha_i+k\delta}]z^k= \exp\left ( (q_i-q_i^{-1})\sum_{k=1}^{\infty}E_{k\delta}^{(i)}z^k\right ).
\end{equation*}

Recall the standard $\mathbb{C}$-algebra antiautomorphism (see \cite{L}) $\Omega:U_q(\mathfrak{g})\to U_q(\mathfrak{g})$ given by $\Omega(E_i)=F_i,\Omega(F_i)=E_i, \Omega(K)=K^{-1}$ for $K\in U_q^0(\mathfrak{g})$, and $\Omega(q^{1/2})=q^{-1/2}$. For $\beta\in \Delta^+$, $F_{\beta}:=E_{-\beta}$ is defined to be $\Omega(E_\beta)$.

We remark that the vectors $E_{k\delta}^{(i)}$ are related to the Drinfeld generators by $E_{k\delta}^{(i)}=\gamma^{-|k|/2}h_{ik}$ for all $k\in \mathbb{Z}_{\neq 0}$ (see \cite{CFM}). Similarly, the vectors $E_{\pm \alpha_i+k\delta},i\in I, k\in \mathbb{Z}$ are related to the Drinfeld generators $x_{i,k}^\pm$ by multiplication with some $K\in U_q^0$.

We recall:
\begin{theorem}[\cite{B}]
The set of monomials in $E_\alpha,\alpha\in \Delta^+$ (resp. $E_\alpha,\alpha\in \Delta^-$) with respect to any ordering is a basis of $U_q^+$ (resp. $U_q^-)$.
\end{theorem}

The total height of $x$ is defined to be $d_0=\sum_{j=1}^k (a_j+a_j')\text{ht}(\gamma_j)$. We will use the following ordering of positive roots:  
\begin{equation}
\beta_0 > \beta_{-1} > \beta_{-2}> \dots >\delta >2\delta>  \dots
>\beta_2 >\beta_1.
\end{equation} 
We also set $-\alpha < -\beta$ if and only if $\beta > \alpha$ for all
positive roots $\alpha, \beta$, providing an ordering
on $\Delta^-$. Beck's total ordering of the positive roots can be divided into three sets:
$$
\{ \alpha + k\delta\ |\ \alpha \in  \mathring{\Delta}^+, k \ge 0\} >
\{k\delta \ |\ k>0\} > \{-\alpha+k\delta\ |\ \alpha \in   \mathring{\Delta}^+, k>0\}.
$$
Similarly, for the negative roots, we have,
$$
\{ -\alpha - k\delta\ |\ \alpha \in \mathring{\Delta}^+, k \ge 0\} <
\{-k\delta \ |\ k>0\} < \{\alpha-k\delta\ |\ \alpha \in   \mathring{\Delta}^+, k>0\}.
$$

For $r\in\mathbb Z$, we define
$$\beta_r^{\pm}=\begin{cases}
\pm\beta_r&{\text{
if}}\ r\leq 0\cr\mp\beta_r&{\text{
if}}\ r> 0.\end{cases}$$
Hence, we have $S=\{\beta_r^+\mid r\in \mathbb Z\}\cup(\Delta^{im})^+$. Set
$$X_{\beta_r^+}=\begin{cases}
 E_{\beta_r}&{\text{
if}}\ r\leq 0\cr-F_{\beta_r}K_{\beta_r}&{\text{
if}}\ r\geq 1,\end{cases}\ \ \ X_{\beta_r^-}=\begin{cases}
 F_{\beta_r}&{\text{
if}}\ r\leq 0\cr-K_{\beta_r}^{-1}E_{\beta_r}&{\text{
if}}\ r\geq 1.
\end{cases}$$
We have the following result which gives a PBW basis associated with a non-standard partition of the root system.

\begin{theorem}[\cite{CFM}]\label{thm-basis}
Given $\um:\mathbb Z\ni r\mapsto m_r\in\mathbb N$ such that $\#\{r\in\mathbb Z|m_r\neq 0\}<\infty$ define
$$X^-(\um)=\prod_{r\in\mathbb Z}X_{\beta_r^-}^{m_r},\ \ X^+(\um)=\prod_{r\in\mathbb Z}X_{\beta_r^+}^{m_r}$$
where one chooses a fixed ordering for the products.

Given $\ul: \mathbb N \times I \to\mathbb N$ such that $\#\{(r,i)\in \mathbb N \times I|l_{(r,i)}\neq 0\}<\infty$ define 
$$
E^{{\text{
im}}}(\ul)=\prod_{(r,i)\in  \mathbb N \times I}E_{(r\delta,i)}^{l_{(r,i)}},\ \ 
F^{{\text{im}}}(\ul)=\Omega(E^{{\text{im}}}(\ul)),
$$
where $E_{(r,i)}=E^{(i)}_{r\delta}$.
Then the set 
\begin{equation}\label{imaginarypbw}
\{X^-(\um)F^{{\text{
im}}}(\ul)K_{\alpha}D^r\gamma^{s/2}E^{{\text{im}}}(\ul^{\prime})X^+(\um^{\prime})\},\quad r,s\in\mathbb Z,\quad \alpha\in \mathring Q
\end{equation}
is a basis of $U_q(\mathfrak g)$. 
\end{theorem}

A \emph{weight $\mathfrak g$-module} $\VV$ 
with respect to $\H$ has a decomposition $\VV = \bigoplus_{\l \in \H^*} \VV_\l$, where
$\VV_\l = \{v \in \VV \mid hv = \l(h)v$ for all $h \in \H\}$,
and we say that the set of weights of $\VV$  is the \emph{support} of $\VV$
and write  $\textrm{supp}(\VV)=\{\lambda\in \H^* \mid  \VV_{\lambda}\neq 0\}$. Similarly, a \emph{weight $U_q(\mathfrak g)$-module} $\VV$ 
 has a decomposition $\VV = \bigoplus_{\l \in P} \VV_\l$, where $\VV_\l = \{v \in \VV \mid K_iv = q^{\l(h_i)}v, D^{\pm 1}v = q^{\pm\l(d)}v\}$.

\subsection{A-forms}
Set $\mathbb A= \mathbb C[q^{1/2},q^{-1/2}, \frac{1}{[n]_{q_i}}, i\in I, n>1]$.   Following \cite{CFM} we define the algebra $U_{\mathbb A}=U_{\mathbb A}(\mathfrak g)$ to be the $\mathbb A$-subalgebra
of $U_q(\mathfrak g)$ with 1 generated by the elements 
$$  
x_{ir}^{\pm 1},\ h_{is}, \ K_i^{\pm 1}, \ \gamma^{\pm 1/2}, D^{\pm 1}, \Lnum{K_i}{s}{r},\Lnum{D}{s}{r},\Lnum{\gamma}{s}{1}_i, \Lnum{\gamma\psi_i}{k,l}{1}
$$
for each $i \in I$, $s \in \mathbb Z$ and $n \in  \mathbb Z_+$, where following \cite{L},
we define the {\it Lusztig elements} in $U_q(\mathfrak g)$:
\begin{align}
\Lnum{\gamma}{s}{1}_i&= \frac{\gamma^{s}-\gamma^{-s}}{q_i-q^{-1}_i},\\
\Lnum{\gamma\psi_i}{k,l}{1}&= \frac{\gamma^{\frac{k-l}{2}}\psi_{i,k+l} -
    \gamma^{\frac{l-k}{2}}\phi_{i,k+l}}{q_i-q^{-1}_i} \\ 
\Lnum{K_i}{s}{r} &=
\prod_{j=1}^r
\frac{K_iq_i^{s-j+1} - K_i^{-1}q_i^{-(s-j+1)}}{q_i^j-q_i^{-j}}, \quad \text{and}\\
\Lnum{D}{s}{r} &= \prod_{j=1}^r
\frac{Dq_0^{s-j+1} - D^{-1}q_0^{-(s-j+1)}}{q_0^j-q_0^{-j}}
\end{align}
where $q_0=q^{d_0}$. 
  Let $U_{\mathbb A}^+$ (resp. $U_{\mathbb A}^-$) denote the subalgebra of
$U_{\mathbb A}$ generated by the $x_{ik}^+$, $k\in\mathbb Z$,  $i \in I$ (resp. $x_{ik}^-$,  $k\in\mathbb Z$,  $i \in I$). Also denote by $U_{\mathbb A}^0$ 
the subalgebra of $U_{\mathbb A}$ generated
by the elements  $h_{il}$,  $l\in\mathbb N\backslash\{0\}$, $1\leq i\leq n$, $\gamma^{\pm 1/2}, K_i^{\pm 1}, \Lnum{K_i}{s}{r}$,
$ D^{\pm 1}, \Lnum{D}{s}{r}$,
$\Lnum{\gamma}{s}{1}_i$ and $ \Lnum{\gamma\psi_i}{k,-k}{1}$.

We recall:
\begin{lemma}  \cite[Proposition 3.10, Lemma 3.15]{B}\label{psix2} Set $a=q_i^2\gamma^{1/2}$, $b=q_i^2\gamma^{-1/2}$, $c=-q_i^{a_{ij}}\gamma^{1/2}$, $d=-q_i^{a_{ij}}\gamma^{-1/2}$ for $i\neq j$, $r>0$, $m\in\mathbb Z$.  Then
\begin{align*}
[\psi_{ir},x_{im}^-]&=-\gamma^{1/2}[2]_i\left(\sum_{k=1}^{r-1}a^{k-1}(q_i-q_i^{-1})\psi_{i,r-k}x_{i,m+k}^-+a^{r-1}x_{i,m+r}^-\right)  \\
[\psi_{ir},x_{im}^+]&=\gamma^{-1/2}[2]_i\left(\sum_{k=1}^{r-1}b^{k-1}(q_i-q_i^{-1})x_{i,m+k}^+\psi_{i,r-k}+b^{r-1}x_{i,m+r}^+\right)  \\ 
[\psi_{ir},x_{jm}^-]&=\gamma^{1/2}[a_{ij}]_i\left(\sum_{k=1}^{r-1}c^{k-1}(q_i-q_i^{-1})\psi_{i,r-k}x_{j,m+k}^-+a^{r-1}x_{j,m+r}^-\right)   \\
[\psi_{ir},x_{jm}^+]&=-\gamma^{-1/2}[a_{ij}]_i\left(\sum_{k=1}^{r-1}d^{k-1}(q_i-q_i^{-1})x_{j,m+k}^+\psi_{i,r-k}+d^{r-1}x_{j,m+r}^+\right) .
\end{align*}
\end{lemma}
A similar set of  formulas can be obtained for $\phi_{ir}$ by applying  the anti-automorphism $\Omega$.

The next result follows from direct calculations, see \cite{CFM}, Proposition 3.4.2.
\begin{proposition} \label{Arelns}
The following commutation relations hold between the generators
of $U_{\mathbb A}$.  For $k\in\mathbb Z$, $l\in\mathbb Z\backslash\{0\}$, $1\leq i,j\leq n$,
\begin{align*}
x_{ik}^+ \Lnum{K_j}{s}{r} & = \Lnum{K_j}{s-a_{ji}}{r}x_{ik}^+, \\
x_{ik}^+ \Lnum{D}{s}{r} &= \Lnum{D}{s-k}{r}x_{ik}^+, \\
\Lnum{K_j}{s}{r}x_{ik}^- &= x_{ik}^- \Lnum{K_j}{s-a_{ji}}{r}, \\
\Lnum{D}{s}{r}x_{ik}^- &= x_{ik}^- \Lnum{D}{s-k}{r},\\
\Lnum{D}{s}{r}h_{ik} &= h_{ik} \Lnum{D}{s+k}{r},\\
[\gamma^{\pm 1/2},U_{\mathbb A} ]&=[D,K_i^{\pm 1}]=[K_i,K_j]=[K_i,h_{jk}]=0, \\
Dh_{ik} &=q^kh_{ik}D,\quad Dx_{ik}^{\pm} =q^kx_{ik}^{\pm}D,\\
K_ix_{jr}^{\pm} &= q_i^{\pm  (\alpha_i|\alpha_j)}x_{jr}^{\pm}K_i,   \\  
[h_{ik},h_{jl}]&=\delta_{k,-l} \frac{1}{k}[ka_{ij}]_i\Lnum{\gamma}{k}{1}_i \\
[h_{ik},x^{\pm}_{jl}]&= \pm \frac{1}{k}[ka_{ij}]_i\gamma^{\mp |k|/2}x^{\pm}_{j,k+l},  \\
[x^+_{ik},x^-_{jl}]&=\delta_{ij}
    \Lnum{\gamma\psi_i}{k,l}{1}.
\end{align*}

\end{proposition}


Applying Proposition~\ref{Arelns} and following the proof of Corollary 3.4.3 in \cite{CFM} we obtain:
\begin{corollary} \label{Arelnscor}The algebra $U_{\mathbb A}$
inherits the  triangular decomposition of $U_q(\mathfrak g)$: 
$U_{\mathbb A}=U_{\mathbb A}^- U_{\mathbb A}^0 U_{\mathbb A}^+$.  
In particular,
any element $u $ of $U_{\mathbb A}$ can be written as a $\mathbb A$-linear combination of
monomials of the form $u^-u^0u^+$ where $u^{\pm} \in U_{\mathbb A}^{\pm}$
and $u^0 \in U_{\mathbb A}^0$.
\end{corollary}

\section{Irreducible weight modules over quantum Heisenberg algebras}

In the Drinfeld presentation (\cite{D}) of the quantized enveloping algebra $U_q(\mathfrak{g})$, the quantum Heisenberg subalgebra $G_q$ of  $U_q(\mathfrak{g})$ is generated over $\mathbb{C}(q^{1/2})$ by $h_{ik}$, $i\in I$, $k\in\mathbb{Z}_{\neq 0}$ and $\gamma^{\pm 1/2}$, 
where $\gamma^{\pm 1/2}$ is central and 
\begin{equation}\label{eq:qheisenberg_rels}
h_{ik}h_{jl}-h_{jl}h_{ik}=\delta_{k+l,0}\cdot 
a_{ij;q}^{k}
\cdot \frac{\gamma^{k}-\gamma^{-k}}{q-q^{-1}},\qquad
\forall i,j\in I, k,l\in\mathbb{Z}_{\neq 0} 
\end{equation}
where
\begin{equation}\label{eq:aijq_def}
a_{ij;q}^k=
\frac{[k\cdot a_{ij}]_{q_i}}{k\cdot [d_j]_{q_j}}.
\end{equation}
%
%
We have
\begin{equation}\label{eq:aij_limit}
\lim_{q\to 1} a_{ij;q}^k=
\frac{\big(\alpha_i|\alpha_j\big)}{d_id_j},
\end{equation}
where $\lim_{q\to 1} a_{ij;q}^k$ means we expand any expression $[n]_{q_i}$ in powers of $q_i$ and then specialize $q$ to 1. By \eqref{eq:aij_limit} and the non-degeneracy of the form $(\cdot |\cdot)$, the matrix $(a_{ij;q}^k)_{i,j\in I}$ is invertible for any $k\in \mathbb{Z}_{\neq 0}$.

Let $(b_{ij;q}^k)_{i,j\in I}$ be the inverse matrix of $(a_{ij;q}^k)_{i,j\in I}$. Make the following change of variables among the negative generators of $G_q$:
\begin{equation}
h_{j,-k}'=\sum_{m=1}^n b_{mj;q}^k h_{m,-k},\quad j\in I, k\in \mathbb{Z}_{>0}.
\end{equation}
Then \eqref{eq:qheisenberg_rels} is equivalent to
($\forall i,j\in I, \forall k,l\in\mathbb{Z}_{>0}$)
\begin{subequations}\label{eq:qheisenberg_newrels}
\begin{align}
h_{ik}h_{j,-l}'-h_{j,-l}'h_{ik} &=\delta_{kl}\delta_{ij}
\cdot \frac{\gamma^k-\gamma^{-k}}{q-q^{-1}}, \\ 
h_{ik}h_{jl}-h_{jl}h_{ik} &=0, \\ 
h_{i,-k}'h_{j,-l}'-h_{j,-l}'h_{i,-k}' &=0.
\end{align}
\end{subequations}

Let $A_{n\cdot \infty}$ be the countably infinite rank Weyl algebra with generators $X_{ik}, \partial_{ik}$, ($i\in I, k\in\mathbb{Z}_{>0}$) and relations
\begin{gather} 
\partial_{ik}X_{ik}-X_{ik}\partial_{ik}=1,\\ 
X_{ik}X_{jl}-X_{jl}X_{ik}=\partial_{ik}X_{jl}-X_{jl}\partial_{ik}=\partial_{ik}\partial_{jl}-\partial_{jl}\partial_{ik}=0,\qquad (i,k)\neq (j,l).
\end{gather}

\begin{proposition} \label{prp:qheisenberg_iso}
For any $\ell\in\mathbb{Z}\setminus\{0\}$, we have an isomorphism of $\mathbb{C}(q^{1/2})$-algebras
\begin{align*}
\psi: G_q/\langle \gamma-q^\ell\rangle &\longrightarrow A_{n\cdot\infty}\\ 
 h_{ik}&\longmapsto [k\ell]_q \cdot \partial_{ik} \\ 
 h_{i,-k}'&\longmapsto X_{ik} 
\end{align*}
for all $i\in I, k\in \mathbb{Z}_{>0}$.
\end{proposition}
\begin{proof}
Using \eqref{eq:qheisenberg_newrels} one checks that $\psi$ is well-defined. It is clearly surjective. Then using the relations of the Weyl algebra one checks that the inverse assignments also define a surjective homomorphism.
\end{proof}

Hence, we can identify $ G_q/\langle \gamma-q^\ell\rangle$ with a countably infinite rank Weyl algebra.  
In \cite{FGM} the authors classified all simple weight (when all $\partial_{ik}X_{ik}$ are simultaneously diagonalizable)
modules over countably infinite rank Weyl algebras. Through the isomorphism in Proposition \ref{prp:qheisenberg_iso} we obtain an ample supply of irreducible weight modules over the quantum Heisenberg algebra $G_q$.

\subsection{$\varphi$-imaginary Verma modules for the quantum Heisenberg algebra}
We construct a particular class of modules over the quantum Heisenberg algebra $G_q$ quantizing the so-called $\varphi$-imaginary Verma modules for classical Heisenberg algebras considered in \cite{BBFK}.
Denote by $\HH_q$ the associative algebra over $\mathbb{C}(q^{1/2})$ generated by $a_i,i\in \mathbb{Z}_{\neq 0}, \gamma^{\pm 1/2}$ with $\gamma^{\pm 1/2}$ central and the following relations:
\begin{equation}\label{eq:untwisted_qheisenberg_rels}
a_ia_j-a_ja_i=\delta_{i+j,0}\frac{[2i]_q}{i}\frac{\gamma^i-\gamma^{-i}}{q-q^{-1}}, i\in \mathbb{Z}_{\neq 0}.
\end{equation}
$\HH_q$ is $\mathbb{Z}$-graded with $\deg(a_i)=i$ and $\deg(\gamma^{\pm 1/2})=0$.
Due to \eqref{eq:qheisenberg_newrels} we have an isomorphism between $G_q$ and the tensor product of $n$ copies of $\HH_q$. Therefore, we can use representations of $\HH_q$ to construct modules over 
$G_q$. Again $\HH_q$ can be viewed as a certain Weyl algebra by Proposition \ref{prp:qheisenberg_iso}.

Let $\varphi:\mathbb{Z}_{>0}\to \{\pm\}$ and define $\HH_\varphi^{\pm}$ to be the subalgebra of $\HH_q$ generated by $a_{\pm\varphi(i)i},i\in \mathbb{N}.$  Let $B_{\varphi}$ be the subalgebra of $\HH_q$ generated by $\HH_\varphi^+$ and $\gamma^{\pm 1/2}.$  Let $\mathbb{C}(q^{1/2})v_\ell,\ell\in \mathbb{Z}$ be the representation of $B_\varphi$ such that $\gamma\cdot v_\ell=q^\ell v_\ell, \HH_\varphi^+\cdot v_\ell=0.$  Then the $\varphi$-imaginary Verma module for $\HH_q, \ell$ is defined to be:
\begin{equation}
\MM_\varphi^q(\ell)=\HH_q\otimes_{B_{\varphi}} \mathbb{C}(q^{1/2})v_\ell.
\end{equation}
We denote by $\HH$ the infinite Heisenberg Lie algebra generated by $a_i,i\in \mathbb{Z}_{\neq 0}, c$ with $c$ central and the following relations:
$$
[a_i,a_j]=\delta_{i+j,0} c, i,j\in \mathbb{Z}_{\neq 0}.
$$
$\varphi$-Imaginary Verma modules for $\HH$ were studied in \cite{BBFK}.

$\HH_q$ has a natural grading in which $a_i, i\in \mathbb{Z}_{\neq 0}$ is given degree $i$ and $c$ is given degree $0$. We define $(\HH_q)_n$ to be the homogeneous component of $\HH_q$ of degree $n$ and set $(\HH_\varphi^-)_n=(\HH_q)_n\cap \HH_\varphi^-$. We have the following analogous properties of $\varphi$-imaginary Verma modules for quantum Heisenberg algebra, $\HH_q$:
\begin{proposition}\label{prop-heis-equi}\begin{enumerate}
\item $\MM_\varphi^q(\ell)$ is a $\mathbb{Z}$-graded $\HH_q$-module, where
 $$\MM_{\varphi}^q(\ell)=\bigoplus_{n\in \mathbb{Z}}\MM_{\varphi}^q(\ell)_n,$$
and $\MM_{\varphi}^q(\ell)_n=(\HH_\varphi^-)_n\cdot v_{\ell}$. If $\varphi(k)\neq \varphi(l)$ for some $k,l\in \mathbb{N}$, then $\MM_\varphi^q(\ell)_n$ is infinite dimensional for any $n\in \mathbb{Z}$,
\item $\MM_\varphi^q(\ell)$ is irreducible if and only if $\ell\neq 0$.
\end{enumerate}
\end{proposition}
\begin{proof}
By Proposition \ref{prp:qheisenberg_iso} we have $\HH_q/\langle\gamma-q^{\ell}\rangle\cong U(\HH)/\langle c-\ell\rangle$. Therefore, there is an equivalence between the category of weight modules of $\HH_q$ on which $\gamma$ acts as $q^\ell$ and the category of modules of $U(\HH)$ on which $c$ acts as $\ell$. The result then follows from Propositions 3.2 and 3.3 of \cite{BBFK}.
\end{proof}
Of course, the construction of $\varphi$-imaginary Verma modules and their properties are easily extendible to the case of $G_q$.

\section{Generalized Loop Modules}
Here we review a construction of generalized loop modules of the Heisenberg subalgebra of an untwisted loop algebra from \cite{BBFK}, \cite{FK1} and \cite{FK2}. They are defined by inducing from certain parabolic submodules based on a non-standard partition of the roots of $\mathfrak g$.

Consider a Heisenberg Lie subalgebra  $G:= \mathbb Cc \oplus \bigoplus_{n \in \mathbb Z \setminus \{0\}} \G_{n\delta}$   of
the affine algebra $\G$.   Thus,
$[x,y] = (x,y)c$ for all $x \in \G_{m\delta}, y \in \G_{n \delta}$,  where
$(x,y)$ is a  skew-symmetric bilinear form such that $(\G_{m\delta}, \G_{n \delta}) =
0$ if $n+m \neq 0$ and its restriction to $\G_{m \delta} \times \G_{-m \delta}$
is nondegenerate for all $m \neq 0$.
Set $G^{\pm} = \bigoplus_{n \in \mathbb N} \G_{\pm n \delta}$.

We will denote by $\mathcal K$ the category of all $\mathbb Z$-graded $G$-modules with the grading compatible with the above grading of $G$. If $V$ is a $G$-module (respectively $\G$-module) with a scalar action of $c$ 
then this scalar is called the {\em level} of $V$.

Let $\PP=\PP_0\oplus \N$ be a  parabolic subalgebra of $\G$
with the Levi subalgebra $\PP_0=G+\H$. If $\NN$ is a module from $\mathcal K$ then one defines on it a structure of a  $\PP_0$-module 
by choosing any $\lambda\in \H^*$ and setting $hv=\lambda(h)v$ for any $h\in \H$ and any $v\in \NN$. Moreover,
$\NN$ can be viewed as a $\PP$-module with a trivial action of
$\N$.

Following \cite{BBFK} and \cite{FK1}, \cite{FK2} consider the induced $\G$-module $${\rm ind}_{\lambda}(\PP, \G; \NN)=U(\G)\otimes_{U(\PP)} \NN.$$ Hence we obtain a functor
 $${\rm ind}_{\lambda}(\PP, \G): \NN\longmapsto {\rm ind}_{\lambda}(\PP, \G; \NN)$$
from the category $\mathcal K$ of $G$-modules to the category of weight $\G$-modules. Let $\tilde{\mathcal K}$ be  the full subcategory of $\mathcal K$ consisting of  modules
on which  $c$ acts injectively. 

The following  theorem  shows that the functor ${\rm ind}_{\lambda}(\PP, \G)$ preserves irreducibility when applied to any irreducible $G$-module in $\tilde{\mathcal K}$.  A particular case of this result was established in \cite{BBFK}.  

\begin{theorem}[\cite{BBFK}, \cite{FK1}, \cite{FK2}]\label{induced-irred}
Let $\PP=\PP_0\oplus \N$ be a  parabolic subalgebra of $\G$, where
$\PP_0=G+\H$. Let $\lambda\in \H^*$ such that $\lambda(c)\neq 0$.   Then the functor ${\rm ind}_{\lambda}(\PP, \G)$ preserves the
irreducibles.
\end{theorem}
In particular, let $\PP=(\H+G)\oplus\G_{R}$ be the parabolic subalgebra of $\G$ with Levi factor $\H+G$, where $\G_{R} = \bigoplus_{\beta \in R} \G_{\beta}$, where $R=\{\alpha+k\delta \mid \alpha\in \mathring{\Delta}^+, k\in \mathbb{Z}\}$.

Let $\VV\in  \tilde{\mathcal K}$ of level $\ell \neq 0$, $\VV = \bigoplus_{k \in \mathbb Z} \VV_k$.  Consider $\lambda\in \H^*$  such that $\lambda(c)=\ell$ and define a $\PP$-module structure on $\VV$ by setting $\G_{R}\VV=0$  and  $hw=(\lambda+ k\delta)(h)w$ for any $w\in \VV_k$ and any $h\in \H$. Then the {\it generalized loop module} is defined as 
\begin{equation}\label{eq:genloop} \MM(\lambda, \VV)=U(\G)\otimes_{U(\PP)}\VV \end{equation}
and we have:
\begin{corollary}\label{cor-irr-gen-loop}
If $\VV$ is an irreducible $G$-module in $\tilde{\mathcal K}$ of level $\lambda(c) \neq 0$, then $\MM(\lambda, \VV)$ is an irreducible $\G$-module.
\end{corollary}

Note that $\textrm{supp }\MM(\lambda,\VV)=\bigcup_{\beta\in \mathring{Q}_+}\{\lambda-\beta+k\delta\mid k\in \mathbb
Z\}$. Moreover, $\dim \MM(\lambda, \VV)_{\mu} =\infty$ for any $\mu$ of the form $\mu=\lambda-\beta+k\delta$ where $\beta \neq 0$ and $k\in \mathbb Z$.

\section{Generalized Loop Modules--The Quantum Case}
Now we define generalized loop modules for the quantized algebra.
Let $\mathcal{K}_q$ denote the category of $\mathbb Z$-graded $G_q$-modules. Let $U_q^d(\pm)$ be the subalgebra of $U_q(\mathfrak g)$
generated by the Drinfeld generators $x^{\pm}_{ik}$ ($i\in  I$, $k\in\mathbb Z$) and $G_q^ d$ be the subalgebra of $U_q(\mathfrak g)$
generated by $G_q$ and $U_q^ 0$.  Set $B_q^d$ to be the subalgebra generated by $U_q^d(+)\cup G_q^ d$. Let $\VV$ be a $G_q$-module on which $c$ acts as the scalar $\ell\in \mathbb{Z}$ and $\lambda\in P$ be a weight such that $\lambda(c)=\ell$. Extend the action of $G_q$ on $\VV$ to $B_q^d$ by setting  $x^+_{ik}\cdot \VV=0$, $K_i^{\pm 1}\cdot v = q^{\pm\lambda(h_i)}v$, and $D^{\pm 1}\cdot v = q^{\pm \lambda(d)}v$ for all $i \in I$, $k\in\mathbb Z$, and $v\in \VV$. Define $\MM_q(\lambda,\VV)=U_q(\mathfrak g)\otimes_{B_q^d} \VV$.  Then $\MM_q(\lambda,\VV)$ is called the {\it quantum generalized loop module} associated with $\VV$ and $\lambda$.

We can also define quantum generalized loop modules using the first presentation of $U_q(\mathfrak g)$. Let $U_q^r(\pm)$ be the subalgebra of $U_q(\mathfrak g)$ generated by $\{ X_{\beta_k^\pm} \mid k\in \mathbb{Z}\}$, and $G_q^r$ be the subalgebra generated by $\{E_{k\delta}^{(i)}\mid i\in I, k \in \mathbb{Z}\} \cup \{\gamma^{\pm 1/2}\}$ and $U_q^ 0$. Note that $G_q^r=G_q^ d$. Denote by $B_q^r$  the subalgebra of $U_q(\mathfrak g)$ generated by $U_q^r(+)\cup G_q^r$. Let $\lambda \in P$ and $\VV$ be an irreducible $G_q^r$-module such that $K_i^{\pm 1}\cdot v = q^{\pm\lambda(h_i)}v$ and $D^{\pm 1}\cdot v = q^{\pm \lambda(d)}v$ for $v\in \VV$ and $i \in I \cup \{0\}$. We give $\VV$ the structure of a $B_q^r$-module by setting $X_{\beta_k^+}\VV=0$. We have:

\begin{proposition}
$\MM_q(\lambda, \VV)$ is isomorphic to  $U_q(\mathfrak g)\otimes_{B_q^r} \VV$ as $U_q$-modules.
\end{proposition}

\begin{proof}
The proof is similar to the proof of Corollary 4.1.3 in \cite{CFM}, which concerns the case where $\VV=\mathbb{C}v$.
\end{proof}

Applying Theorem \ref{thm-basis} we immediately obtain:

\begin{theorem}\label{LoopSpan}
Suppose $\VV$ is an irreducible module in $\mathcal{K}_q$.  As a vector space over $\mathbb{C}(q^{1/2})$, $\MM_q(\lambda,\VV)$ is isomorphic to $\bigoplus_i U_q(-)\otimes v_i$ where $\{v_i\}$ is a basis of $\VV$, hence is free as a $U_q(-)$-module.
\end{theorem}

Denote by $G_{\mathbb{A}}^d$ the ${\mathbb{A}}$-subalgebra of $U_\mathbb{A}(\mathfrak{g })$ generated by 
$$ h_{is}, \ K_i^{\pm 1}, \ \gamma^{\pm 1/2}, D^{\pm 1}, \Lnum{K_i}{s}{r},\Lnum{D}{s}{r},\Lnum{\gamma}{s}{1}_i, \Lnum{\gamma\psi_i}{k,l}{1}
$$
for $i \in I,r,s\in\mathbb Z, s\neq 0$. Using Proposition \ref{prop-heis-equi} one constructs a large supply of  irreducible $G_{\mathbb{A}}$-modules.	
We have $G_{\mathbb{A}}^d/\langle q^{1/2}-1\rangle\simeq U(G+\mathfrak{h})$ and $G_{\mathbb{A}}^d\otimes_\mathbb{A} \mathbb{C}(q)\simeq G_q^d$.

Let $\VV$  be  an irreducible $G_q$-module in $\mathcal{K}_q$. Extend it to a $G_q^d$-module through some fixed $\lambda\in P$. Suppose there exists a basis 
$\{v_j\}$, $j\in J$ of $\VV$ such that $V_{\mathbb{A}}=\sum_j \mathbb{A}v_j$ is a $G_{\mathbb{A}}^d$-submodule such that its classical limit $V_{\mathbb{A}}/(q-1)V_{\mathbb{A}}$ is an irreducible $G+\mathfrak{h}$-module in the category $\mathcal{K}$. We will call such module $\mathbb{A}$-{\em admissible}.
Again using Proposition \ref{prop-heis-equi}   $\mathbb{A}$-admissible irreducible $G_q$-modules can be constructed from irreducible modules of a certain Weyl algebra.

Now fix  an irreducible module $\VV\in \mathcal{K}_q$ which is $\mathbb{A}$-admissible with respect to a suitable basis $\{v_j\}$, $j\in J$. We define the $\mathbb{A}$-form of $\MM_q(\lambda,\VV)$ to be the $U_\mathbb{A}(\mathfrak{g})$-submodule $\MM^{\mathbb{A}}(\lambda,\VV)=\sum_j U_{\mathbb{A}}(\mathfrak{g})\otimes_{B^q} v_j$.  We immediately have the following:
\begin{lemma}\label{LoopAForm}
As a module over $\mathbb{A}$, $\MM^{\mathbb{A}}(\lambda,\VV)$ is spanned by $U^-_{\mathbb{A}}\otimes v_j$, and is free as a $U_\mathbb{A}^-$-module.
\end{lemma}

Consider the map $\xi:\mathbb{C}(q^{1/2})\otimes_\mathbb{A} \MM^{\mathbb{A}}(\lambda,\VV)\to \MM_q(\lambda,\VV)$ given by $\xi(f \otimes v)=fv,f\in \mathbb{C}(q^{1/2}), v\in \MM^\mathbb{A}(\lambda,\VV).$   This is clearly surjective, and has inverse $\zeta:\MM_q(\lambda,\VV)\to\mathbb{C}(q^{1/2})\otimes_\mathbb{A} \MM^{\mathbb{A}}(\lambda,\VV)$ given by $\zeta(X^-(\um)v_j)=1\otimes X^-(\um) v_j,$ where $\{X^-(\um)v_j\}$ is an $\mathbb{A}$-basis of $\MM_q(\lambda,\VV)$.  Therefore, we have the following: 
\begin{proposition}\label{LoopIso}
For any $\lambda\in P$ and any $\mathbb{A}$-{\em admissible} $\VV\in \mathcal{K}_q$,  $\mathbb{C}(q^{1/2})\otimes_\mathbb{A} \MM^{\mathbb{A}}(\lambda,\VV)\cong \MM_q(\lambda,\VV)$ as $\mathbb{C}(q^{1/2})$ vector spaces.
\end{proposition}
From Lemmas \ref{LoopSpan} and \ref{LoopAForm} we have:
\begin{proposition}
$\MM^\mathbb{A}(\lambda,\VV)$ is a weight module with the weight decomposition $\MM^\mathbb{A}(\lambda,\VV)=\bigoplus_{\mu \in P}\MM^\mathbb{A}(\lambda,\VV)_\mu$, where $\MM^\mathbb{A}(\lambda,\VV)_\mu=\MM^\mathbb{A}(\lambda,\VV)\cap \MM_q(\lambda,\VV)_\mu$.
\end{proposition}
The isomorphism given in Proposition \ref{LoopIso} restricts to the weight spaces to give that
for each $\mu \in P$, $\MM(\lambda,\VV)_\mu$ is a free $\mathbb{A}$-module. Also, when $\dim_{\mathbb{C}(q^{1/2})}(\MM_q(\lambda,\VV)_\mu)< \infty$, we have $\textup{rank}_\mathbb{A}(\MM^\mathbb{A}(\lambda,\VV)_\mu)=\dim_{\mathbb{C}(q^{1/2})}(\MM_q(\lambda,\VV)_\mu).$

We now give the classical limits of $U_q(\mathfrak{g})$ and $\MM_q(\lambda,\VV).$  We set $U'=\mathbb{A}/\mathbb{J}\otimes_\mathbb{A} U_\mathbb{A},$ where $\mathbb{J}$ is the ideal of $\mathbb{A}$ generated by $q^{1/2}-1$. Now, define $\overline{U}=U'/K'$ where $K'$ is the ideal of $U'$ generated by $K_i-1,D-1,$ and $\gamma^{1/2}-1$. We denote by $\overline{u}$ the image of $u\in U'$ under this isomorphism. Thus, $\overline{U}$ is the $q^{1/2}=1$ limit of $U_q(\G)$, hence $\overline{U}\cong U(\mathfrak{g})$, with $\overline{E}_i,\overline{F}_i,\overline{\lbinom{K_i;r}{s}},i\in I\cup\{0\}$ and $\overline{\lbinom{D;r}{s}}$ getting sent to $e_i,f_i,\binom{h_i+r}{s}, i\in I\cup \{0\},$ and $\binom{d+r}{s}$ respectively.  Also, we have that the root generators of $U_\mathbb{A}$, namely $\{x_{ik},h_{il}| i\in I, k\in \mathbb{Z},l\in \mathbb{Z}_{\neq 0}\}$ are sent to the corresponding root generators of $U(\mathfrak{g})$ under this isomorphism (see \cite[Theorem 4.7]{B}).

For $\lambda\in P$ and $V$ an $\mathbb{A}$-admissible irreducible module in $\mathcal{K}_q$ let $\overline{\MM}(\lambda,\VV)=\mathbb{A}/\mathbb{J}\otimes_\mathbb{A} \MM^{\mathbb{A}}(\lambda,\VV).$ After taking this classical limit, $\VV$ becomes an irreducible $G$-module $\overline{\VV}$ in $\mathcal{K}$, since it is assumed to be $\mathbb{A}$-admissible. For $\mu\in P,$ let $\overline{\MM}(\lambda,\VV)_\mu=\mathbb{A}/\mathbb{J}\otimes_\mathbb{A} \MM^{\mathbb{A}}(\lambda,\VV)_\mu.$  Thus, $\overline{\MM}(\lambda,\VV)$ is a $U'$-module and the weight space decomposition of $\MM^\mathbb{A}(\lambda, \VV)$ gives $\overline{\MM}(\lambda,\VV)=\bigoplus_{\mu\in P}\overline{\MM}(\lambda,\VV)_\mu$.  Also $\dim_{\mathbb{C}}(\overline{\MM}(\lambda,\VV)_\mu)=\textup{rank}_{\mathbb{A}}(\MM^{\mathbb{A}}(\lambda,\VV)).$ We have that $K_i,D,$ and $\gamma^{1/2}$ act as the identity on $\overline{\MM}(\lambda,\VV)$, since they act as integral powers of $q^{1/2}$ on weight vectors of $\MM^\mathbb{A}(\lambda,\VV)$, which evaluate to $1$ in the classical limit. Therefore we have that $\overline{\MM}(\lambda,\VV)$ is also a well-defined $\overline{U}$-module, called the classical limit of $\MM^\mathbb{A}(\lambda,\VV).$ Using this identification one may easily verify the following:
\begin{proposition}\label{Classical3}
Let $\VV\in\mathcal{K}_q$ be an irreducible $\mathbb{A}$-admissible $G_q$-module.  Then as a $U(\mathfrak{g})$-module, $\overline{\MM}(\lambda,\VV)$ is a weight module generated by $\overline{\VV}$ such that, for any $\mu \in P, \overline{\MM}(\lambda,\VV)_\mu$ is the $\mu$-weight space of $\overline{\MM}(\lambda,\VV).$ 
\end{proposition}
Finally, we prove:
\begin{proposition}\label{Classical4}
For any $\mathbb{A}$-admissible irreducible $G_q$-module in $\mathcal{K}_q$ and any  $\lambda \in P$, 
$\overline{\MM}(\lambda,\VV)$ is a free $U(\mathfrak{g}_{-R})$-module generated by $\VV$, where $R=\{\alpha+k\delta \mid \mathring{\Delta}, k\in \mathbb{Z}\}$.
\end{proposition}
\begin{proof}
Let $\{v_j\}$, $j\in J$ be a suitable basis of $V$. Then $\MM^\mathbb{A}(\lambda,\VV)$ is the $\mathbb{A}$-span of $u^-\otimes_\mathbb{A} v_{j},$ for $u^-\in U^-_{\mathbb{A}}$, by Lemma \ref{LoopAForm}. Therefore, by the previous proposition, $\overline{\MM}(\lambda,\VV)$ is spanned over $\mathbb{C}$ by the images $\overline{u^-}\otimes_\mathbb{C} \overline{v}_{j}.$ But the images $\overline{u^-}\in \overline{U}$ correspond to basis monomials in $U(\mathfrak{g}_{-R}),$ and the $\overline{v_j}$ are a basis of $\overline{\VV}$ by the admissibility of $\VV$.  Therefore, $\overline{\MM}(\lambda,\VV)$ is a free $U(\mathfrak{g}_{-R})$-module, generated by $\overline{v_j},j\in J$.
\end{proof}
The conclusion is that $\overline{\MM}(\lambda,\VV)\cong\MM(\lambda,\overline{\VV})$. By Propositions \ref{Classical3} and \ref{Classical4} we see that $\MM_q(\lambda,\VV)$ is a  quantum deformation of $\MM(\lambda,\overline{\VV})$.  Hence, applying Corollary \ref{cor-irr-gen-loop} we obtain our main result which allows to construct irreducible $U_q(\mathfrak g)$-modules from irreducible $G_q$-modules.

\begin{theorem}\label{LoopMain}
If $\VV\in\mathcal{K}_q$ is an irreducible $\mathbb{A}$-admissible $G_q$-module of level $\lambda(c)\neq 0$, then $\MM_q(\lambda,\VV)$ is irreducible $U_q(\mathfrak g)$-module.
\end{theorem}
\begin{proof}
Let $\WW$ be a proper submodule of $\MM_q(\lambda,\VV)$. Then $\WW^\mathbb{A}=\WW \cap \MM^{\mathbb{A}}(\lambda,\VV)$ is a submodule of $\MM^\mathbb{A}(\lambda,\VV)$ and the classical limit of $\WW^\mathbb{A}$ is a proper submodule of $\MM(\lambda,\VV)$. But, $\MM(\lambda,\VV)$ is irreducible by  \cite{BBFK} (Theorem 5.6), \cite{FK2}. Therefore, so is $\MM_q(\lambda,\VV).$
\end{proof}
We believe the restriction on $\VV$ to be $\mathbb{A}$-admissible can be lifted.

\section{$\varphi$-imaginary Verma modules for affine quantum algebras}
In this section we consider a particular class of quantum generalized loop modules. In the $q=1$ case they were introduced in \cite{BBFK}. Let $\varphi: \mathbb N \rightarrow \{+, -\}$ be an arbitrary function. Denote
$S_\varphi=(S\cap \Delta^{\text{re}})\cup \{n\delta|n\in \mathbb{Z}_{>0},\varphi(n)=+\}\cup \{-m\delta|m\in \mathbb{Z}_{>0},\varphi(m)=-\}.$  Let $\lambda\in P,$ $B_{q}^\varphi$ be the subalgebra of $U_q(\mathfrak{g})$ generated by $\{E_\beta | \beta\in S_\varphi\}\cup U_q^0(\mathfrak{g})$.  Let $\mathbb{C}(q^{1/2})v_\lambda$ be the representation of $B_q^\varphi$ such that $K_i\cdot v_\lambda=q^{\lambda(h_i)}v_\lambda,i\in I,D\cdot v_\lambda=q^{\lambda(d)}v_\lambda, U_q(S_{\varphi})\cdot v_\lambda=0.$  Then the $\varphi$-imaginary Verma module for $U_q(\mathfrak{g})$ is defined to be:
\begin{equation}
\MM^\varphi_q(\lambda)=U_q(\mathfrak{g})\otimes_{B^{\varphi}_q} \mathbb{C}(q^{1/2}) v_\lambda.
\end{equation}
We have the following property of $\MM^\varphi_q(\lambda),$ which is an analogue of the $U(\mathfrak{g})$ case in \cite{BBFK}:
\begin{proposition}[Analogous to Proposition 3.4 in \cite{BBFK}]
If $\varphi(k)\neq \varphi(l)$ for some $k,l\in \mathbb{Z}_{>0}$, then $\textup{dim }\MM^\varphi_q(\lambda)_\mu=\infty$ for all weights $\mu$.
\end{proposition}
Consider the $G_q$-submodule of $\MM^\varphi_q(\lambda)$ generated by $v_\lambda$: $\NN_q=G_q v_\lambda$. Then 
\begin{proposition}\label{Heis-quantum-irr}
The module $\NN_q$ is irreducible in $\mathcal{K}_q$  if and only if $\lambda(c)\neq 0$.
Moreover, it is an $\mathbb{A}$-{\em admissible} module.
\end{proposition}

\begin{proof}
Note that $\NN_q$ is a $\varphi$-highest weight module, i.e. for any $n\in\mathbb N$ either $\mathfrak g_{n\delta}v_\lambda=0$ or $\mathfrak g_{-n\delta}v_\lambda=0$ depending on the sign of $\varphi(n)$.
Then a standard argument shows that $\NN_q$ is irreducible when $\lambda(c)\neq 0$. If $\lambda(c)= 0$ then the maximal submodule of $\NN_q$ is the vector space direct sum complement of $\mathbb{C}(q^{1/2}) v_\lambda.$
Choose a $\varphi$-highest weight vector $v\in \NN_q$ and let $\NN_{\mathbb{A}}=G_{\mathbb{A}}^d v$.
Then the classical limit of $\NN_{\mathbb{A}}$ is irreducible because the universal $\phi$-highest weight module is irreducible itself. Therefore $\NN_q$ is $\mathbb{A}$-admissible.
\end{proof}
We have an isomorphism
$$\MM^\varphi_q(\lambda)\simeq \MM_q(\lambda, \NN_q).$$
\begin{theorem}\label{teor-phi}
Let  $\varphi: \mathbb N \rightarrow \{+, -\}$, $\lambda\in P$.  Then  $\MM^\varphi_q(\lambda)$ is irreducible if and only if 
$\lambda(c)\neq 0$.
\end{theorem}

\begin{proof}
Indeed, if $\lambda(c)\neq 0$ then $\NN_q$ is irreducible by Proposition~\ref{Heis-quantum-irr}. Applying Theorem \ref{LoopMain} we immediately obtain that $\MM^\varphi_q(\lambda)$ is irreducible. Conversely, suppose $\MM^\varphi_q(\lambda)$ is irreducible but $\lambda(c)= 0$.
Then $\NN_q$ contains a nontrivial submodule by Proposition~\ref{Heis-quantum-irr}, which in turn generates a nontrivial submodule of 
$\MM^\varphi_q(\lambda)$. 
\end{proof}

\section{Acknowledgments}
 The first author  was supported in part by the CNPq grant (301743/2007-0) and by the
Fapesp grant (2010/50347-9). The third author is grateful to the University of S\~{a}o Paulo for hospitality, his supervisor Vyacheslav Futorny, and to the Fapesp for financial support (grant number 2011/12079-5).

\end{document}